\documentclass[12pt,reqno]{article}

\usepackage{amsfonts}
\usepackage{mathrsfs}

\usepackage{amsmath,amsthm,amsfonts,amssymb,latexsym,mathrsfs,color}

\usepackage[colorlinks=true,
linkcolor=webblue, filecolor=webbrown, citecolor=webred ]{hyperref}
\definecolor{webblue}{rgb}{0,.5,0}
\definecolor{webred}{rgb}{0,.5,0}
\definecolor{webbrown}{rgb}{.6,0,0}

\newtheorem{thm}{Theorem}[section]
\newtheorem{lem}[thm]{Lemma}

\newtheorem{prop}[thm]{Proposition}

\theoremstyle{definition}

\newtheorem{rem}[thm]{Remark}

\numberwithin{equation}{section}

\newcommand{\des}{{\rm des\,}}

\newcommand{\msn}{{\mathcal S}_n}

\title{A unified approach to combinatorial triangles: a generalized Eulerian
polynomial
\author{Bao-Xuan Zhu\thanks{Supported by the National Natural Science Foundation of China (No. 11971206).
\newline\hspace*{5mm}
   {\it Email address:} bxzhu@jsnu.edu.cn (B.-X. Zhu)}}}
\date{\footnotesize School of Mathematics and Statistics,
         Jiangsu Normal University,
         Xuzhou 221116, PR China}

\begin{document}

\maketitle

\begin{abstract}
Motivated by the classical Eulerian number, descent and excedance
numbers in the hyperoctahedral groups, an triangular array from
staircase tableaux and so on, we study a triangular array $[\mathcal
{T}_{n,k}]_{n,k\ge 0}$ satisfying the recurrence relation:
\begin{equation*}
\mathcal {T}_{n,k}=\lambda(a_0n+a_1k+a_2)\mathcal
{T}_{n-1,k}+(b_0n+b_1k+b_2)\mathcal
{T}_{n-1,k-1}+\frac{cd}{\lambda}(n-k+1)\mathcal {T}_{n-1,k-2}
\end{equation*}
with $\mathcal {T}_{0,0}=1$ and $\mathcal {T}_{n,k}=0$ unless $0\le
k\le n$. We derive a functional transformation for its
row-generating function $\mathcal{T}_n(x)$ from the row-generating
function $A_n(x)$ of another array $[A_{n,k}]_{n,k}$ satisfying a
two-term recurrence relation. Based on this transformation, we can
get properties of $\mathcal {T}_{n,k}$ and $\mathcal{T}_n(x)$
including nonnegativity, log-concavity, real rootedness, explicit
formula and so on. Then we extend the famous Frobenius formula, the
$\gamma$ positivity decomposition and the David-Barton formula for
the classical Eulerian polynomial to those of a generalized Eulerian
polynomial. We also get an identity for the generalized Eulerian
polynomial with the general derivative polynomial.

Finally, we apply our results to an array from the Lambert function,
a triangular array from staircase tableaux and the alternating-runs
triangle of type $B$ in a unified approach.
\end{abstract}


\section{Introduction}
\subsection{Eulerian numbers }

Let $S_n$ denote the symmetric group on $n$-elements
$[n]=\{1,2,\ldots,n\}$. Let $\pi=\pi(1)\cdots \pi(n)\in S_n$. An
element $i\in [n-1]$ is called a descent of $\pi$ if
$\pi(i)>\pi({i+1})$. The classical Eulerian polynomial is defined by
$$E_n(x)=\sum_{\pi\in S_n}x^{1+\des(\pi)}=\sum_{k}\left\langle
  \begin{array}{ccccc}
    n \\
   k\\
  \end{array}
\right\rangle x^k,$$ where $\des(\pi)$ denotes the number of
descents of $\pi$ and $\left\langle
  \begin{array}{ccccc}
    n \\
   k\\
  \end{array}
\right\rangle$ is called the {\it Eulerian number}. Eulerian numbers
and polynomials have many nice properties. For example:
\begin{itemize}
\item [\rm (i)]
The Eulerian number $\left\langle
  \begin{array}{ccccc}
    n \\
   k\\
  \end{array}
\right\rangle$ satisfies the recurrence relation
\begin{equation*}\label{rr-an}
\left\langle
  \begin{array}{ccccc}
    n \\
   k\\
  \end{array}
\right\rangle=k\left\langle
  \begin{array}{ccccc}
    n-1 \\
   k\\
  \end{array}
\right\rangle+(n-k+1)\left\langle
  \begin{array}{ccccc}
    n-1 \\
   k-1\\
  \end{array}
\right\rangle
\end{equation*}
with initial conditions $\left\langle
  \begin{array}{ccccc}
    0 \\
   0\\
  \end{array}
\right\rangle=1$ and $\left\langle
  \begin{array}{ccccc}
    0 \\
   k\\
  \end{array}
\right\rangle=0$ for $k\geq1$ or $k<0$, see \cite{Com74}. Clearly,
if let $\left\langle
  \begin{array}{ccccc}
    n+1 \\
   k+1\\
  \end{array}
\right\rangle=\overrightarrow{E}_{n,k}$ for $n,k\geq1$, then shift
Eulerian numbers $\overrightarrow{E}_{n,k}$ satisfy the next
recurrence
\begin{equation}\label{rec+shift+Eulerian}
\overrightarrow{E}_{n,k}=(k+1)\overrightarrow{E}_{n-1,k}+(n-k+1)\overrightarrow{E}_{n-1,k-1}
\end{equation}
 with initial conditions $\overrightarrow{E}_{0,0}=1$ and $\overrightarrow{E}_{0,k}=0$
for $k\geq1$ or $k<0$.
\item [\rm (ii)]
It is well-known that $E_n(x)$ is related to Stirling numbers of the
second kind $\left\{
  \begin{array}{ccccc}
    n \\
   k\\
  \end{array}
\right\}$ by the famous Frobenius formula
\begin{eqnarray}\label{relation+stirling+eulerian}
E_{n}(x) =\sum_{i=1}^ni!\left\{
  \begin{array}{ccccc}
    n \\
   i\\
  \end{array}
\right\}x^i(1-x)^{n-i}.
\end{eqnarray}
This implies an explicit formula for Eulerian numbers:
\begin{eqnarray}\label{formu+eulerian}
\left\langle
  \begin{array}{ccccc}
    n \\
   k\\
  \end{array}
\right\rangle&=&\sum_{i=1}^n\left\{
  \begin{array}{ccccc}
    n \\
   i\\
  \end{array}
\right\}i!\binom{n-i}{k-i}(-1)^{k-i},
\end{eqnarray}
where
\begin{eqnarray}\label{formu+stirling}
\left\{
  \begin{array}{ccccc}
    n \\
   i\\
  \end{array}
\right\}i!=\sum_{j=1}^i(-1)^{i-j}\binom{i}{j}j^n
\end{eqnarray}
for $n,i\geq0$, see B\'{o}na \cite{Bona12} for instance.
\item [\rm (iii)]
It is also well known that $E_n(x)$ has only real zeros and
therefore is log-concave. In addition, Foata and Sch\"{u}tzenberger
\cite{FS70} and Foata and Strehl \cite{FS74} proved for $E_n(x)$
that there exists nonnegative numbers $\gamma_{n,k}$ such that
\begin{equation}\label{relation+gamma+Eulerian}
E_n(x)=\sum_{k=0}^{n/2}\gamma_{n,k}x^k(1+x)^{n-1-2k}
\end{equation} for $n\geq0$. This is an important property (called the {\it $\gamma$-positivity }\cite{At17}) because it implies symmetry and unimodality of
$E_n(x)$.
\item [\rm (iv)]
The Eulerian polynomial $E_n(x)$ also has a closely connection with
the famous alternating-runs polynomial $R_n(x)$ by
\begin{eqnarray}\label{relation+run+Eulerian}
R_n(x) = \left(\frac{1+x}{2}\right)^{n-1}
(1+w)^{n+1}E_n\left(\frac{1-w}{1+w}\right), w =\sqrt{\frac{1 - x}{1
+ x}},\, \text{for}\, n\geq2,
\end{eqnarray} see David and Barton \cite{DB} and Knuth \cite{Kur73}.
\end{itemize}
The Eulerian number and polynomial have a rich history and appear in
many contexts in combinatorics; see \cite{Pe15} for a detailed
exposition.

\subsection{Descent and excedance numbers in the hyperoctahedral groups}

Let $B_n$ be the hyperoctahedral group of rank $n$. Let
$\left\langle
  \begin{array}{ccccc}
    n \\
   k\\
  \end{array}
\right\rangle_B$ be the Eulerian number of type $B$ counting the
elements of $B_n$ with $k$ $B$-descents. Then $\left\langle
  \begin{array}{ccccc}
    n \\
   k\\
  \end{array}
\right\rangle_B$ satisfies the recurrence
\begin{equation}\label{rec-Euler+B}
\left\langle
  \begin{array}{ccccc}
    n \\
   k\\
  \end{array}
\right\rangle_B=(2k+1)\left\langle
  \begin{array}{ccccc}
    n-1 \\
   k\\
  \end{array}
\right\rangle_B+(2n-2k+1)\left\langle
  \begin{array}{ccccc}
    n-1 \\
   k-1\\
  \end{array}
\right\rangle_B
\end{equation}
with initial conditions $\left\langle
  \begin{array}{ccccc}
    0 \\
   0\\
  \end{array}
\right\rangle_B=1$ and $\left\langle
  \begin{array}{ccccc}
    n \\
   k\\
  \end{array}
\right\rangle_B=0$ unless $0\leq k\leq n$, see Brenti
\cite{Bre94EuJC}. Let $ \mathfrak{B}_n(x)=\sum_{k=0}^{n}\left\langle
  \begin{array}{cc}
    n \\
   k\\
  \end{array}
\right\rangle_Bx^k $ be the Eulerian polynomials of type $B$. In
fact, Eulerian numbers (resp. polynomials) of type $B$ share many
properties of classical Eulerian numbers (resp. polynomials), see
\cite[A060187]{Slo} for details.

The excedance number in $S_n$ is defined by:
$$exc(\pi)=|\{i\in\{1,2,\ldots,n\}: \pi(i)>i,\pi\in S_n\}|.$$ It is
well-known that the excedance number and the descent number have the
same distribution and they both are the classical Eulerian number
$\left\langle
  \begin{array}{ccccc}
    n \\
   k\\
  \end{array}
\right\rangle$.

Bagno {\it et al.} gave a definition of the excedance number in
$B_n$, see \cite{BGMS15} for details. Let $X_{n,k}$ be the number of
permutations $\pi\in B_n$ having $exc_A(\pi)=k$. It was proved that
$$X_{n,k}=(n-k)X_{n-1,k-1}+(n+1)X_{n-1,k}+(k+1)X_{n-1,k+1}$$
for $n\geq1$ with $X_{0,0}=1$ \cite{BGMS15}. In addition, in
\cite{BGMS15}, they showed that the sequence $(X_{n,k})_{k=0}^n$ is
log-concave by induction. Let $X^*_{n,k}=X_{n,n-k}$. Then the array
$[X^*_{n,k}]_{n,k\geq0}$ satisfies the recurrence relation
\begin{eqnarray}\label{rec+X}
X^*_{n,k}=kX^*_{n-1,k}+(n+1)X^*_{n-1,k-1}+(n-k+1)X^*_{n,k-2}
\end{eqnarray} for $n\geq1$ with $X^*_{0,0}=1$. Using the recurrence relation (\ref{rec+X}), we proved many
properties of $X^{*}_{n,k}$ similar to those of the Eulerian number
$\left\langle
  \begin{array}{ccccc}
    n \\
   k\\
  \end{array}
\right\rangle$, see \cite[Propostion 3.3]{Zhu20} for details.

\subsection{Staircase tableaus}

The Partially Asymmetric Simple Exclusion Process (PASEP) is a
physical model in which $n$ sites on a one-dimensional lattice are
either empty or occupied by a single particle. This model is also a
Markov chain. Staircase tableaux are combinatorial objects which
appear as key tools in the study of the PASEP physical model, see
Corteel, Mandelshtam and Williams \cite{CMW17}, Corteel, Stanley,
Stanton, and Williams \cite{CSSW11}, Corteel and Williams
\cite{CW07,CW10,CW11} for instance. Let us recall the definitions as
follows.

 A staircase tableau $T$ of size $n$ is a
Ferrers diagram of ¡°staircase¡± shape $(n, n-1, \cdots, 2, 1)$ such
that boxes are either empty or labeled with $\alpha$, $\beta$,
$\gamma$ or $\delta$, and satisfying the following conditions: (1)
no box along the diagonal of $T$ is empty; (2) all boxes in the same
row and to the left of a $\beta$ or a $\delta$ are empty; (3) all
boxes in the same column and above an $\alpha$ or a $\gamma$ are
empty.

For staircase tableaux without any $\gamma$ label, we denote by
$\mathscr{T}_{n,k}$ the number of such tableaux $T$ of size $n$ with
$k$ labels $\alpha$ or $\delta$ in the diagonal. Aval, Boussicault
and Dasse-Hartaut \cite{ABD13}  proved for $n\geq0$ and all $k\in Z$
that
\begin{eqnarray}\label{rec+ST}
\mathscr{T}_{n,k}&=&(k+1)\mathscr{T}_{n-1,k}+(n+1)\mathscr{T}_{n-1,k-1}+(n-k+1)\mathscr{T}_{n-1,k-2},
\end{eqnarray}
where $\mathscr{T}_n(x)=\sum_{k\geq0}\mathscr{T}_{n,k}q^k$ for
$n\geq0$. In \cite{ABD13}, it was proved that the polynomial
$\mathscr{T}_n(x)$ has only real zeros and is log-concave.

\subsection{An array from the Lambert function} The Lambert $W$ function
was defined in \cite{CGHJK} by $$We^W=x.$$ The $n$-th derivative of
$W$ is given implicitly by
$$\frac{d^nW(x)}{dx^n}=\frac{e^{-nW(x)}p_n(W(x))}{(1+W(x))^{2n-1}}$$
for $n\geq1$, where
$p_n(x)=(-1)^{n-1}\sum_{k=0}^{n-1}\beta_{n,k}x^k$ are polynomials
satisfying the recurrence relation
\begin{eqnarray*}
p_{n+1}(x)=-(nx+3n-1)p_n(x)+(1+x)p'_n(x)
\end{eqnarray*}
for $n\geq1$, where the array $[\beta_{n,k}]_{n\geq1,k\geq0}$
satisfies the recurrence relation
\begin{eqnarray}\label{rec+lab+array}
\beta_{n+1,k}=(3n-k-1)\beta_{n,k}+n\beta_{n,k-1}-(k+1)\beta_{n,k+1}
\end{eqnarray}
for $n,k\geq0$ with $\beta_{1,0}=1$. Kalugin and Jeffrey \cite{KJ11}
proved that if each polynomial $(-1)^{n-1}p_n(w)$ has all positive
coefficients, then it follows that $\frac{dW(x)}{dx}$ is a
completely monotonic function. In order to get the positivity, by
nontrivial computations, they proved that the coefficients are
unimodal and log-concave. If let
$\beta^{\circ}_{n,k}=\beta_{n+1,n-k}$ for $n,k\geq0$, then the array
$[\beta^{\circ}_{n,k}]_{n,k\geq0}$ exactly satisfies the recurrence
relation
\begin{eqnarray}\label{rec+labert+array}
\beta^{\circ}_{n,k}=n\beta^{\circ}_{n-1,k}+(2n+k-1)\beta^{\circ}_{n-1,k-1}-(n-k+1)\beta^{\circ}_{n-1,k-2}
\end{eqnarray}
for $n,k\geq0$ with $\beta^{\circ}_{0,0}=1$, which easily implies
the positivity of $\beta^{\circ}_{n,k}$, see Section $4$.

\subsection{Structure of this paper}
As we know that Eulerian polynomials and Eulerian numbers have many
nice properties. In addition, different generalizations of Eulerian
polynomials and Eulerian numbers were considered, see Haglund-Zhang
\cite{HZ19}, Han-Mao-Zeng \cite{HMZ19}, Rz\c{a}dkowski-Urli\'{n}ska
\cite{RU19}, Zhu \cite{Zhu20}, Zhuang \cite{Zhuang17}. But recently
in many new combinatorial enumerations, there bring out more and
more combinatorial triangles satisfying certain three-term
recurrence similar to recurrence relations (\ref{rec+X}),
(\ref{rec+ST}) and (\ref{rec+labert+array}). Motivated by these, the
aim of this paper is to consider their generalization and extend
those properties of Eulerian polynomials and Eulerian numbers to
those of the generalized case. Then we can apply to some
combinatorial triangles in a unified approach.

Let $\mathbb{R}$ (resp. $\mathbb{R^{+}}$, $\mathbb{R^{\geq}}$) be
the set of all (resp., positive, nonnegative) real numbers. Let
$\{b_1,c\}\in \mathbb{R}$, $\lambda\in \mathbb{R^{+}}$ and
$\{a_0,a_1,a_2,b_0,b_2,d\}\subseteq \mathbb{R^{\geq}}$. Define a
generalized triangular array $[\mathcal {T}_{n,k}]_{n,k\ge 0}$ by
the recurrence relation:
\begin{equation}\label{Recurece+T2}
\mathcal {T}_{n,k}=\lambda(a_0n+a_1k+a_2)\mathcal
{T}_{n-1,k}+(b_0n+b_1k+b_2)\mathcal
{T}_{n-1,k-1}+\frac{cd}{\lambda}(n-k+1)\mathcal {T}_{n-1,k-2}
\end{equation}
with $\mathcal {T}_{0,0}=1$ and $\mathcal {T}_{n,k}=0$ unless $0\le
k\le n$. Also denote its row-generating function by $\mathcal
{T}_n(x)=\sum_{k\geq0}\mathcal {T}_{n,k}x^k$ for $n\geq0$.

In Section $2$, we derive a functional transformation of
$\mathcal{T}_n(x)$ from the row-generating function of another array
$[A_{n,k}]_{n,k}$ satisfying a two-term recurrence relationship, see
Theorem \ref{thm+m+tran}. Based on the transformation, we can
immediately apply nonnegativity, log-concavity and zeros of $A_n(x)$
to those of $\mathcal {T}_n(q)$.

In Section $3$, we first give a generalization of the formula
(\ref{formu+stirling})(Theorem \ref{thm+Stirling+family}). Then we
extend the identity (\ref{relation+stirling+eulerian}) and formula
(\ref{formu+eulerian}) to the generalized Frobenius formula (Theorem
\ref{thm+EF}). We also present a $\gamma$ positivity result for a
generalized Eulerian polynomial (Theorem \ref{thm+SEF}), which gives
a unified proof for $\gamma$ positivity of Eulerian polynomials of
two kinds. Finally, for a generalized Eulerian polynomial, we
present a David-Barton type formula, which can be looked as a
generalization of (\ref{relation+run+Eulerian}).

In Section $4$, we apply our results to some important combinatorial
triangles in a unified manner, including the array
$[\beta_{n,k}]_{n,k}$ from the Lambert function, the triangular
array $[\mathscr{T}_{n,k}]_{n,k}$ from staircase tableaux, and
alternating-runs triangle of type $B$.

Throughout this paper, for an array $[M_{n,k}]_{n,k\geq0}$, we let
its row-generating function $M_n(x)=\sum_{k=0}^nM_{n,k}x^k$ for
$n\geq0$ and its reciprocal array $[M^*_{n,k}]_{n,k\geq0}$ by
$M^*_{n,k}=M_{n,n-k}$.

\section{An important functional transformation of
$\mathcal{T}_n(x)$}

In this section, we will first present a functional transformation
of $\mathcal{T}_n(x)$ from a triangular array satisfying a two-term
recurrence.

\begin{thm}\label{thm+m+tran}
 Let $[\mathcal {T}_{n,k}]_{n,k\ge 0}$ be defined
in (\ref{Recurece+T2}). If $b_1=da_1-c$, then there exists an array
$[A_{n,k}]_{n,k\geq0}$ satisfying the recurrence relation
 \begin{eqnarray}\label{rec+A+tran}
A_{n,k}
=(a_0n+a_1k+a_2)A_{n-1,k}+[[b_0+d(a_1-a_0)]n-(c+da_1)k+b_2+d(a_1-a_2)]A_{n-1,k-1}
\end{eqnarray}
with $A_{0,0}=1$ and $A_{n,k}=0$ unless $0\le k\le n$ such that
their row-generating functions satisfy
\begin{eqnarray}\label{rel+T+A}
\mathcal {T}_n(x)=(\lambda+dx)^nA_n(\frac{x}{\lambda +dx})
\end{eqnarray} for $n\geq0$. In addition,
$$\mathcal {T}_{n,k}=\sum_{i\geq0}A_{n,i}\binom{n-i}{k-i}{\lambda}^{n-k}d^{k-i}$$
for $n\geq0$ and $k\geq0$.
\end{thm}

\begin{proof}
Because the array $[A_{n,k}]_{n,k\geq0}$ satisfies the recurrence
relation
 \begin{eqnarray*}
A_{n,k}
=(a_0n+a_1k+a_2)A_{n-1,k}+[[b_0+d(a_1-a_0)]n-(c+da_1)k+b_2+d(a_1-a_2)]A_{n-1,k-1}
\end{eqnarray*}
with $A_{0,0}=1$ and $A_{n,k}=0$ unless $0\le k\le n$, its
reciprocal array $[A^*_{n,k}]_{n,k\geq0}$ satisfies the recurrence
relation
\begin{eqnarray*}
A^*_{n,k}
&=&[(b_0-da_0-c)n+(c+da_1)k+b_2+d(a_1-a_2)]A^*_{n-1,k}+[(a_0+a_1)n-a_1k+a_2]A^*_{n-1,k-1}
\end{eqnarray*}
with $A^*_{0,0}=1$ and $A^*_{n,k}=0$ unless $0\le k\le n$. It
follows for $n\geq1$ that its row-generating functions satisfy the
next recurrence
\begin{eqnarray}\label{rec+A*(x)}
A^*_n(x)
&=&[(b_0-da_0-c)n+b_2+d(a_1-a_2)+x[(a_0+a_1)n-a_1+a_2]]A^*_{n-1}(x)+\nonumber\\
&&x(c+da_1-a_1x)(A^{*}_{n-1}(x))^{'}.
\end{eqnarray}
This implies that
\begin{eqnarray*}
A^*_n(\lambda x+d)
&=&\left[(b_0-da_0-c)n+b_2+d(a_1-a_2)+(\lambda x+d)[(a_0+a_1)n-a_1+a_2]\right]A^*_{n-1}(\lambda x+d)\nonumber\\
&&+(\lambda x+d)(c-a_1\lambda x)A^{*'}_{n-1}(\lambda x+d).
\end{eqnarray*}
If let $B_n(x)=A^*_n(\lambda x+d)$ for $n\geq0$, then for $n\geq1$,
we have
\begin{eqnarray*}
B_n(x) &=&\{(b_0+da_1-c)n+b_2+\lambda
x[(a_0+a_1)n-a_1+a_2]\}B_{n-1}(x)+(\lambda x+d)(c-\lambda a_1
x)\frac{[B_{n-1}(x)]^{'}}{\lambda},
\end{eqnarray*}
which implies that coefficients array $[B_{n,k}]_{n,k\geq0}$ of
$B_n(x)$ satisfies the next recurrence relation
\begin{eqnarray*}
B_{n,k} &=&[(b_0+da_1-c)n+(c-da_1)k+b_2]B_{n-1,k}+\lambda
[(a_0+a_1)n-a_1k+a_2]B_{n-1,k-1}\\
&&+\frac{cd}{\lambda}(k+1)B_{n-1,k+1}
\end{eqnarray*}
for $n\geq1$. Thus its reciprocal array $[B^*_{n,k}]_{n,k\geq0}$
satisfies the next recurrence relation
\begin{eqnarray*}
B^*_{n,k} &=&[b_0n+(da_1-c)k+b_2]B^*_{n-1,k-1}+\lambda
(a_0n+a_1k+a_2)B^*_{n-1,k}\\
&&+\frac{cd}{\lambda}(n-k+1)B^*_{n-1,k-2}
\end{eqnarray*}
with $B^*_{0,0}=1$ and $B^*_{n,k}=0$ unless $0\le k\le n$.

Obviously, if $b_1=da_1-c$, then arrays $[B^*_{n,k}]_{n,k\geq0}$ and
$[\mathcal {T}_{n,k}]_{n,k\geq0}$ satisfy the same recurrence
relations and the same initial conditions. Thus $B^*_{n,k}=\mathcal
{T}_{n,k}.$

In addition, it follows from
\begin{eqnarray}\label{eq+recip}
\mathcal {T}^*_n(x)=B_n(x)=A_n^*(\lambda x+d)=(\lambda
x+d)^nA_n(\frac{1}{\lambda x+d})
\end{eqnarray}
that
\begin{eqnarray*}
\mathcal {T}_n(x)=x^nT^*_n(1/x)=x^n\left(\frac{\lambda}{
x}+d\right)^nA_n\left(\frac{1}{\frac{\lambda}{
x}+d}\right)=(\lambda+dx)^nA_n\left(\frac{x}{\lambda +dx}\right).
\end{eqnarray*}

Finally, using the identity $\mathcal {T}^{*}_n(x)=A^*_n(\lambda
x+d),$ we have
 \begin{eqnarray*}
\mathcal
{T}^*_{n,k}&=&\sum_{i\geq0}A_{n,n-i}\binom{i}{k}{\lambda}^kd^{i-k},
\end{eqnarray*}
which implies $
\mathcal{T}_{n,k}=\sum_{i\geq0}A_{n,i}\binom{n-i}{n-k}{\lambda}^{n-k}d^{k-i}.
$ This completes the proof.
\end{proof}

\begin{rem}\label{rem+T+A}
The formula in (\ref{rel+T+A}) plays an important role. If we know
certain property of $A_n(x)$, then by (\ref{rel+T+A}) we immediately
get that of $\mathcal{T}_n(x)$.
\end{rem}
 Let
$\{a_k\}_{k=0}^n$ be a sequence of positive numbers. It is called
{\it log-concave} if $a_{k}a_{k+2}\le a_{k+1}^2$ for all $k\ge 0$.
It is well-known that if the polynomial $\sum_{k=0}^na_kx^k$ has
only real zeros, then $\{a_k\}_{k=0}^n$ is log-concave. In fact,
log-concave sequences arise often in combinatorics and have been
extensively investigated. We refer the reader to Br\"and\'en
\cite{Bra15}, Brenti \cite{Bre94}, Liu and Wang \cite{LW-RZP},
Stanely \cite{Sta89}, and Wang and Yeh \cite{WY07} for the
log-concavity and real rootedness.

\begin{prop}\label{prop+m+tran}
Let $[\mathcal {T}_{n,k}]_{n,k\ge 0}$ and $[A_{n,k}]_{n,k\ge 0}$ be
defined in (\ref{Recurece+T2}) and (\ref{rec+A+tran}), respectively.
Assume that $b_1=da_1-c$ and the triangle $[A_{n,k}]_{n,k}$ is
positive, then we have
\begin{itemize}
\item [\rm (i)]
each row-sequence $\{\mathcal {T}_{n,k}\}_{k=0}^n$ is positive and
log-concave;
\item [\rm (ii)] if $A_n(x)$ has only real zeros,
then so does $\mathcal {T}_n(x)$ and zeros of $\mathcal {T}_n(x)$
locate in $[-\frac{\lambda}{d},0]$.
\end{itemize}
\end{prop}

\begin{proof}
(i)  By (\ref{rec+A+tran}), it follows from Kurtz \cite[Theorem
2]{Kur72} that each row $\{A_{n,k}\}_{k=0}^n$ is log-concave. Thus
the reciprocal sequence $A_{n,n}$,$A_{n,n-1}$, $\ldots, A_{n,1},
A_{n,0}$ is log-concave. This obviously implies that $A^*_n(dx)$ is
log-concave. Let $g(x):=A^*_n(dx)$ and $f(x):=g(x+1)$. Then by
\cite[Corollary 8.3]{Bre94}(it says that if coefficients of a
polynomial $P(x)$ is positive and log-concave, then so is $P(x+1)$),
we have $f(x)=g(x+1)$ is log-concave. So is $f(\frac{\lambda
}{d}x)$. Hence we get that $\mathcal {T}^*_n(x)$ is log-concave
since $\mathcal {T}^*_n(x)=A^*_n(\lambda x+d)=f(\frac{\lambda
}{d}x)$ in (\ref{eq+recip}). This implies that
 $\mathcal {T}_{n,0},\mathcal {T}_{n,1},\ldots,\mathcal {T}_{n,n}$ is log-concave.

(ii) If $A_n(x)$ has only real zeros, then all zeros of $A_n(x)$ is
non-positive. Thus by (\ref{rel+T+A}), all zeros of $A_n(x)$ is
non-positive and locate in $[-\frac{\lambda}{d},0]$. The proof is
complete.
\end{proof}

Since $A^{*}_n(x)$ satisfies the recurrence relation
(\ref{rec+A*(x)}), for real rootedness of $A_n(x)$, we can use the
following result a special case of Corollary $2.4$ in Liu and Wang
\cite{LW-RZP}.
\begin{lem}\cite{LW-RZP}\label{lem-LW}
Let $\{P_n(x)\}_{n\geq0}$ be a sequence of polynomials with
nonnegative coefficients and $0\leq\deg P_n-\deg P_{n-1}\leq1$.
Suppose that
$$P_n(x)=(a_nx+b_n)P_{n-1}(x)+x(c_nx+d_n)P'_{n-1}(x)$$
where $a_n,b_n\in\mathbb{R}$ and $c_n\le 0,d_n\ge 0$. Then $P_n(x)$
has only real zeros and the zeros of $P_n(x)$ interlace those of
$P_{n-1}(x)$ for $n\geq1$.
\end{lem}

By Remark \ref{rem+T+A}, we will explore other properties of
$A_{n,k}$ and apply to those of $\mathcal {T}_{n,k}$.

\section{Extensions of properties of Eulerian numbers}
In this section, we will generalize some properties of Eulerian
numbers.
\subsection{A generalized Frobenius formula}
In what follows we will generalize the famous Frobenius formula
(\ref{relation+stirling+eulerian}) and (\ref{formu+eulerian}). Let
$F_{n,i}=\left\{
  \begin{array}{ccccc}
    n \\
   i\\
  \end{array}
\right\}i!$ for $n,i\geq0$. It is known that the triangle
$[F_{n,k}]_{n,k\geq0}$ satisfies the recurrence relation
\begin{eqnarray}\label{rel+Frob}
F_{n,k}=kF_{n-1,k}+kF_{n-1,k-1}
\end{eqnarray}
with $F_{0,0}=1$. In convenience, we define a {\it generalized
Frobenius triangle} $[\mathcal {F}_{n,k}]_{n,k}$ by
\begin{eqnarray*}
\mathcal {F}_{n,k}&=&(a_1k+a_2)\mathcal
{F}_{n-1,k}+(b_1k+b_2)\mathcal {F}_{n-1,k-1},
\end{eqnarray*}
with $\mathcal {F}_{0,0}=1$ and $\mathcal{F}_{n,k}=0$ unless $0\le
k\le n$.

As a generalization of (\ref{formu+stirling}), we have the next
result.
\begin{thm}\label{thm+Stirling+family}
Let $\{a_2,b_2\}\subseteq \mathbb{R}$ and $\{a_1,b_1\}\subseteq
\mathbb{R^{+}}$. Then we have
 the explicit formula
\begin{eqnarray*}
\mathcal
{F}_{n,k}=\binom{\frac{b_2}{b_1}+k}{k}\left(\frac{b_1}{a_1}\right)^k\sum_{j=0}^k\binom{k}{j}(-1)^{k-j}(a_2+a_1j)^n
\end{eqnarray*}
for $n,k\geq0$.
\end{thm}

\begin{proof}

Let the exponential generating function
$$F(q,t)=\sum_{n,k\geq0}\mathcal {F}_{n,k}q^k\frac{t^n}{n!}.$$
Then by the recurrence relation of $\mathcal {F}_{n,k}$, we have the
partial differential equation
$$F_t-(b_1q+a_1)qF_q=[a_2+(b_1+b_2)q]F$$
with the initial condition $F(q,0)=1$. Solve this equation and get
\begin{eqnarray}\label{EGF+F}
F(q,t)=e^{a_2t}\left[1+\frac{b_1q(1-e^{a_1t})}{a_1}\right]^{-(1+\frac{b_2}{b_1})}
\end{eqnarray}
(It is also routine to check that $F(q,t)$ is a solution of the
above equation with the initial condition). Then we have
\begin{eqnarray*}
 \sum_{n,k\geq0}
 \mathcal {F}_{n,k}q^k\frac{t^n}{n!}&=&e^{a_2t}\left[1+\frac{b_1q(1-e^{a_1t})}{a_1}\right]^{-(1+\frac{b_2}{b_1})}\\
 &=&e^{a_2t}\left[1-\frac{b_1q(e^{a_1t}-1)}{a_1}\right]^{-(1+\frac{b_2}{b_1})}\\
 &=&e^{a_2t}\sum_{i\geq0}\binom{\frac{b_2}{b_1}+i}{i}\left[\frac{b_1q(e^{a_1t}-1)}{a_1}\right]^i\\
 &=&e^{a_2t}\sum_{i\geq0}\binom{\frac{b_2}{b_1}+i}{i}\left(\frac{b_1q}{a_1}\right)^i\sum_{j=0}^i\binom{i}{j}(-1)^{i-j}e^{a_1jt}\\
 &=&\sum_{i\geq0}\binom{\frac{b_2}{b_1}+i}{i}\left(\frac{b_1q}{a_1}\right)^i\sum_{j=0}^i\binom{i}{j}(-1)^{i-j}e^{(a_2+a_1j)t}\\
 &=&\sum_{i\geq0}\binom{\frac{b_2}{b_1}+i}{i}\left(\frac{b_1q}{a_1}\right)^i\sum_{j=0}^i\binom{i}{j}(-1)^{i-j}\sum_{n\geq0}\frac{(a_2+a_1j)^nt^n}{n!}\\
 &=&\sum_{n,k\geq0}\sum_{j=0}^k\binom{\frac{b_2}{b_1}+k}{k}\left(\frac{b_1}{a_1}\right)^k\binom{k}{j}(-1)^{k-j}(a_2+a_1j)^nq^k\frac{t^n}{n!},
\end{eqnarray*}
which implies
\begin{eqnarray*}
 \mathcal {F}_{n,k}=\binom{\frac{b_2}{b_1}+k}{k}\left(\frac{b_1}{a_1}\right)^k\sum_{j=0}^k\binom{k}{j}(-1)^{k-j}(a_2+a_1j)^n.
\end{eqnarray*}
This completes the proof.
\end{proof}

By Theorem \ref{thm+m+tran} for $c=0$ and Theorem
\ref{thm+Stirling+family}, we get a generalized Frobenius formula as
follows. Obviously, it generalizes the identity
(\ref{relation+stirling+eulerian}) and formula
(\ref{formu+eulerian}).

\begin{thm}\label{thm+EF}
Let $\{a_1,b_1\}\subseteq \mathbb{R^{+}}$ and $\{a_2,b_2\}\subseteq
\mathbb{R^{\geq}}$. If an array $[\mathcal{D}_{n,k}]_{n,k}$
satisfies the recurrence relation:
\begin{eqnarray}\label{rec+EF}
\mathcal{D}_{n,k}&=&(a_1k+a_2)\mathcal{D}_{n-1,k}+(b_1n-b_1k+b_2)\mathcal{D}_{n-1,k-1},
\end{eqnarray}
where $\mathcal{D}_{0,0}=1$ and $\mathcal{D}_{n,k}=0$ unless $0\le
k\le n$, then there exists an array $[\mathcal
{F}_{n,k}]_{n,k\geq0}$ satisfying the recurrence relation
$$\mathcal {F}_{n,k}=(a_1k+a_2)\mathcal {F}_{n-1,k}+(b_1k+b_2-b_1+\frac{a_2}{a_1}b_1)\mathcal {F}_{n-1,k-1}$$
for $n\geq1$ with $\mathcal {F}_{0,0}=1$ such that the row
generating functions
$$\mathcal{D}_n(q)=(1-\frac{b_1}{a_1}q)^n\mathcal
{F}_n\left(\frac{q}{1-\frac{b_1}{a_1}q}\right).$$ In addition, we
have the explicit formula
\begin{eqnarray*}
\mathcal{D}_{n,k}&=&\sum_{i\geq0}\mathcal
{F}_{n,i}\binom{n-i}{k-i}(-\frac{b_1}{a_1})^{k-i},\\
 \mathcal
{F}_{n,i}&=&\sum_{j\geq0}\binom{\frac{a_2}{a_1}+\frac{b_2}{b_1}-1+i}{i}\binom{i}{j}\left(\frac{b_1}{a_1}\right)^{i}(-1)^{i-j}(a_2+a_1j)^n.
\end{eqnarray*}
\end{thm}

\subsection{A generalization of $\gamma$ positivity of Eulerian
polynomials} In this subsection, we will generalize
(\ref{relation+gamma+Eulerian}). In order to do so, we define a
generalized Eulerian triangle $[\mathcal {E}_{n,k}]_{n,k}$ by
\begin{eqnarray}\label{rec+SEF}
\mathcal{E}_{n,k}&=&(a_1k+a_2)\mathcal{E}_{n-1,k}+(a_1n-a_1k+a_2)\mathcal{E}_{n-1,k-1},
\end{eqnarray}
with $\mathcal {E}_{0,0}=1$ and $\mathcal {E}_{n,k}=0$ unless $0\le
k\le n$. Its row-generating function has the following
$\gamma$-positivity decomposition.

\begin{thm}\label{thm+SEF}
Let $\{a_1,a_2\}\subseteq \mathbb{R^{+}}$ and the generalized
Eulerian triangle $[\mathcal{E}_{n,k}]_{n,k}$ is defined in
(\ref{rec+SEF}). Then there exists a nonnegative triangle $[\mathcal
{S}_{n,k}]_{n,k}$ satisfying the recurrence relation
\begin{equation}\label{S+recurrence relation}
\mathcal {S}_{n,k}=(a_1k+a_2)\mathcal {S}_{n-1,k}+c(n-2k+1)\mathcal
{S}_{n-1,k-1}
\end{equation}
with $k\leq \frac{n+1}{2}$ and $\mathcal {S}_{0,0}=1$ such that
their row-generating functions
\begin{eqnarray*}
\mathcal {E}_n(x)=(1+x)^n\mathcal {S}_n\left(\frac{\frac{2a_1}{c}
x}{(1+x)^{2}}\right)
\end{eqnarray*}
for $n\geq1$, where $c>0$.
\end{thm}
\begin{proof}
We will prove this result by induction on $n$. Let $\lambda=2a_1/c$.
It is obvious for $n=0$. For $n=1$, $\mathcal {E}_1(x)=a_2+a_2x$ and
$\mathcal {S}_1(x)=a_2$. Therefore we have
$$\mathcal{E}_1(x)=(1+x)\mathcal {S}_1\left(\frac{\lambda x}{(1+x)^{2}}\right).$$ In the following we
assume $n\geq2$. By the inductive hypothesis, we have
\begin{eqnarray}
\mathcal{E}'_{n-1}(x)=(n-1)(1+x)^{n-2}\mathcal
{S}_{n-1}\left(\frac{\lambda
x}{(1+x)^{2}}\right)+\frac{(1+x)^{n-1}\lambda
(1-x)}{(1+x)^3}\mathcal {S}'_{n-1}\left(\frac{\lambda
x}{(1+x)^{2}}\right).
\end{eqnarray}
On the other hand, by the recurrence relations (\ref{S+recurrence
relation}) and (\ref{rec+SEF}), we have
$$\mathcal {S}_n(x)=[a_2+\frac{2a_1(n-1)x}{\lambda}]\mathcal {S}_{n-1}(x)+x(a_1-\frac{4a_1}{\lambda}x)\mathcal {S}'_{n-1}(x)$$
and
$$\mathcal{E}_n(x)=[a_2+x(a_1n-a_1+a_2)]\mathcal{E}_{n-1}(x)+a_1x(x-1)\mathcal{E}'_{n-1}(x).$$
Then we derive that
\begin{eqnarray*}
\mathcal{E}_n(x) &=&[a_2(x+1)+a_1x(n-1)](1+x)^{n-1}\mathcal
{S}_{n-1}\left(\frac{\lambda
x}{(1+x)^{2}}\right)+a_1x(x-1)\\
&&\times\left[(n-1)(1+x)^{n-2}\mathcal {S}_{n-1}\left(\frac{\lambda
x}{(1+x)^{2}}\right)+\frac{(1+x)^{n-1}\lambda
(1-x)}{(1+x)^3}\mathcal {S}'_{n-1}\left(\frac{\lambda x}{(1+x)^{2}}\right)\right]\\
&=&(1+x)^n\left[a_2+\frac{a_1(n-1)x}{1+x}\right]\mathcal
{S}_{n-1}\left(\frac{\lambda
x}{(1+x)^{2}}\right)+a_1x(1-x)(1+x)^n\\
&&\times\left[\frac{(n-1)}{(1+x)^{2}}\mathcal
{S}_{n-1}\left(\frac{\lambda x}{(1+x)^{2}}\right)+\frac{\lambda
(1-x)}{(1+x)^4}S'_{n-1}\left(\frac{\lambda
x}{(1+x)^{2}}\right)\right]\\
&=&(1+x)^n\left[a_2+\frac{2a_1(n-1)}{\lambda}\frac{\lambda
x}{(1+x)^2}\right]\mathcal {S}_{n-1}\left(\frac{\lambda
x}{(1+x)^{2}}\right)+(1+x)^n\times\frac{a_1}{\lambda}\times
\frac{\lambda
x}{(1+x)^{2}}\times\\
&&\left[\lambda-4\frac{\lambda x}{(1+x)^{2}}\right]\times \mathcal
{S}'_{n-1}\left(\frac{\lambda
x}{(1+x)^{2}}\right)\\
&=&(1+x)^n\mathcal {S}_{n}\left(\frac{\lambda x}{(1+x)^{2}}\right).
\end{eqnarray*}
Finally, obviously, the nonnegativity of $[\mathcal
{S}_{n,k}]_{n,k}$ follows from its recurrence relation
(\ref{S+recurrence relation}). The proof is complete.
\end{proof}

\begin{rem}
By recurrence relations (\ref{rec+shift+Eulerian}) and
(\ref{rec-Euler+B}), respectively, it is clear that Theorem
\ref{thm+SEF} gives a unified proof for $\gamma$ positivity of
Eulerian polynomials of two kinds.
\end{rem}

\subsection{An identity with the general derivative
polynomial}

In this subsection, we will prove a formula of the generalized
Eulerian polynomial $\mathcal {E}_n(x)$ in terms of the general
derivative polynomial.

Let $x=\tan{t}$ and $y=\sec{t}$. Let $D_t$ denote the differential
operator. Obviously, $D_t(x)=y^2$ and $D_t(y)=xy$. For
$\delta\geq1$, define
\begin{eqnarray}\label{def+Q}
\frac{d^n}{dt^n} y^\delta=Q_n(\delta,t)y^\delta.
\end{eqnarray}
 The $Q_n(\delta,t)$ is called {\it the general
derivative polynomial} in $t$ in \cite{V14} and satisfies the
recurrence relation
$$Q_{n+1}(\delta,t)=(1+t^2)D_tQ_n(\delta,t)+\delta tQ_n(\delta,t)$$ with $Q_0(\delta,t)=1$.
Obviously, if $deg(Q_n(\delta,x))$ is even (resp., odd) then so is
the degree of each term of $Q_n(\delta,t)$. Thus we can assume that
$$Q_n(\delta,t):=\sum_{k=0}^{\lfloor \frac{n}{2}\rfloor}Q_{n,k}(\delta)t^{n-2k}.$$
It is also known in \cite{V14} that the exponential generating
function
\begin{eqnarray*}
\sum_{n\geq0}Q_n(\delta,x)\frac{z^n}{n!}=\frac{1}{(\cos z-x \sin
z)^\delta}.
\end{eqnarray*} In addition, $Q_n(1,1)=S_n$,
where $S_n$ is the Springer number which is a type $B$ analog of the
Euler number for the group of signed permutation and counts also the
number of Weyl chambers in the principal Springer cone of the
Coxeter group $B_n$ of symmetries of an $n$ dimensional cube, see
\cite[A001586]{Slo}.

In order to get a connection between the generalized Eulerian
polynomial $\mathcal {E}_n(x)$ and the general derivative polynomial
$Q_n(\delta,x)$, we prove another identity for $Q_n(\delta,x)$ as
follows.
\begin{thm}\label{thm+Andr}
Let $\{a,b,c\}\subseteq \mathbb{R^{+}}$. Assume that $[\mathcal
{S}_{n,k}]_{n,k\geq0}$ is an array satisfying the recurrence
relation:
\begin{equation}\label{recurrence relation}
\mathcal {S}_{n,k}=(a\,k+b)\mathcal {S}_{n-1,k}+c(n-2k+1)\mathcal
{S}_{n-1,k-1}
\end{equation}
with $\mathcal {S}_{0,0}=1$ and $\mathcal {S}_{n,k}=0$ unless $0\le
k\le (n+1)/2$. Then for $n\geq0$, we have their row-generating
functions
$$\mathcal {S}_n(q)=\left(\frac{a}{2}\right)^n\left(\frac{2c}{a}q-1\right)^{n/2}Q_n\left(\frac{2b}{a},\frac{1}{\sqrt{\frac{2c}{a}q-1}}\right).$$
\end{thm}
\begin{proof}
Suppose that an array $[S^\circ_{n,k}]_{n,k\geq0}$ satisfies the
recurrence relation
\begin{equation*}
S^\circ_{n,k}=(2k+\frac{2b}{a})S^\circ_{n-1,k}+(n-2k+1)S^\circ_{n-1,k-1}
\end{equation*}
with  $S^\circ_{0,0}=1$ and $S^\circ_{n,k}=0$ unless $0\le k\le
\frac{n+1}{2}$. It is clear for $n,k\geq0$ that $$\mathcal
{S}_{n,k}=\left(\frac{a}{2}\right)^{n-k}c^kS^\circ_{n,k},$$ which
implies
$$\mathcal {S}_n(q)=\left(\frac{a}{2}\right)^nS^\circ_n(\frac{2cq}{a}).$$ Thus,
in order to show $\mathcal
{S}_n(q)=\left(\frac{a}{2}\right)^n(\frac{2c}{a}q-1)^{n/2}Q_n(\frac{2b}{a},\frac{1}{\sqrt{\frac{2c}{a}q-1}})$
for $n\geq0$, it suffices to prove
$$S^\circ_n(q)=(q-1)^{n/2}Q_n(\delta,\frac{1}{\sqrt{q-1}})$$ with $\delta=\frac{2b}{a}$ for $n\geq0$.

Let $x=\tan{t}$ and $y=\sec{t}$. Then we have
\begin{eqnarray*}
\frac{d}{dt} y^{\delta}&=&\delta xy^{\delta},\\
\frac{d^2}{dt^2} y^\delta&=&(\delta^2x^2+\delta y^{2})y^\delta,\\
\frac{d^3}{dt^3}
y^\delta&=&[\delta^3x^3+(3\delta^2+2\delta)xy^{2}]y^\delta,\\
&\vdots&.
\end{eqnarray*}
Thus we can also assume that
\begin{eqnarray}
\frac{d^n}{dt^n} y^\delta=y^\delta\sum_{k=0}^{\lfloor
n/2\rfloor}S^\star_{n,k}x^{n-2k}y^{2k}.
\end{eqnarray}
By induction on $n$, we get that the array
$[S^\star_{n,k}]_{n,k\geq0}$ satisfies the recurrence relation
\begin{equation*}
S^\star_{n,k}=(2k+\delta)S^\star_{n-1,k}+(n-2k+1)S^\star_{n-1,k-1}
\end{equation*}
with $S^\star_{n,k}=0$ unless $0\le k\le \frac{n+1}{2}$ and
$S^\star_{0,0}=1$. Clearly, $S^\star_{n,k}=S^\circ_{n,k}$. Then
combining
$$x^nS^\star_n(y^2/x^2)=Q_n(\delta,x)$$
and the identity $1+x^2=y^2$, we have
$$S^\star_n(q)=(q-1)^{n/2}Q_n(\delta,\frac{1}{\sqrt{q-1}})$$
for $q>1$ and $n\geq0$. The general derivative polynomial
$Q_n(\delta,t)$ satisfies the recurrence relation
$$Q_{n+1}(\delta,t)=(1+t^2)D_tQ_n(\delta,t)+\delta tQ_n(\delta,t)$$ with $Q_0(\delta,t)=1$.
It is obvious that if $deg(Q_n(\delta,t))$ is even (resp., odd) then
so is the degree of each term of $Q_n(\delta,t)$. This implies that
$(q-1)^{n/2}Q_n(\delta,\frac{1}{\sqrt{q-1}})$ is a polynomial in
$q$. Thus we have polynomials
$$S^\star_n(q)=(q-1)^{n/2}Q_n(\delta,\frac{1}{\sqrt{q-1}})$$
for $n\geq0$. The proof is complete.
\end{proof}

\begin{thm}\label{thm+D+Euler}
Assume that the generalized Eulerian polynomial $\mathcal{E}_{n}(x)$
is defined in (\ref{rec+SEF}) and $Q_n(\delta,x)$ is the general
derivative polynomial. If $0<a_1\leq2a_2$, then we have
\begin{eqnarray}
\mathcal{E}_n(x)=\left(\frac{a_1\sqrt{-1}}{2}\right)^n(1-x)^nQ_n\left(\frac{2a_2}{a_1},\frac{\sqrt{-1}}{w^2}\right),
\end{eqnarray} where $w=\sqrt{\frac{1-x}{1+x}}$.
\end{thm}
\begin{proof}
For $[\mathcal {E}_{n,k}]_{n,k\geq0}$, by Theorem \ref{thm+SEF},
there exists a nonnegative triangle $[\mathcal {S}_{n,k}]_{n,k}$
satisfying the recurrence relation
\begin{equation}\label{rec+SS}
\mathcal {S}_{n,k}=(a_1k+a_2)\mathcal {S}_{n-1,k}+c(n-2k+1)\mathcal
{S}_{n-1,k-1},
\end{equation}
where $\mathcal {S}_{0,0}=1$ and $\mathcal {S}_{n,k}=0$ unless $0\le
k\le (n+1)/2$ such that for $n\geq1$ their row-generating functions
\begin{eqnarray*}
\mathcal {E}_n(x)=(1+x)^n\mathcal {S}_n\left(\frac{\frac{2a_1}{c}
x}{(1+x)^{2}}\right).
\end{eqnarray*}

Then by Theorem \ref{thm+Andr}, we have its row-generating function
\begin{eqnarray}\label{rel+Q+S}
\mathcal
{S}_n(x)=\left(\frac{a_1}{2}\right)^n\left(\frac{2cx}{a_1}-1\right)^{\frac{n}{2}}Q_n\left(\frac{2a_2}{a_1},\frac{1}{\sqrt{\frac{2cx}{a_1}-1}}\right)
\end{eqnarray}
for $n\geq0$.

Thus we have
\begin{eqnarray}
\mathcal
{E}_n(x)=\left(\frac{a_1\sqrt{-1}}{2}\right)^n(1-x)^nQ_n\left(\frac{2a_2}{a_1},\frac{1+x}{\sqrt{-1}(1-x)}\right).
\end{eqnarray}
This proof is complete.
\end{proof}

\begin{rem}
For Eulerian polynomials of two kinds in Section $1$, by recurrence
relations (\ref{rec+shift+Eulerian}) and (\ref{rec-Euler+B}),
respectively, using Theorem \ref{thm+D+Euler} we immediately get
\begin{eqnarray}
\overrightarrow{E}_n(x)=\left(\frac{\sqrt{-1}}{2}\right)^n(1-x)^nQ_n\left(2,\frac{\sqrt{-1}}{w^2}\right),\\
\mathfrak{B}_n(x)=\left(\sqrt{-1}\right)^n(1-x)^nQ_n\left(1,\frac{\sqrt{-1}}{w^2}\right),
\end{eqnarray}
where $w=\sqrt{\frac{1-x}{1+x}}$.
\end{rem}

\subsection{A generalization of formula of David and Barton}
In what follows we will give a generalization of the formula
(\ref{relation+run+Eulerian}) of David and Barton. Let us recall
alternating-runs polynomials $R_n(x)$. For $\pi=\pi(1)\pi(2)\cdots
\pi(n)\in\msn$, we say that $\pi$ changes direction at position $i$
if either $\pi({i-1})<\pi(i)>\pi(i+1)$, or
$\pi(i-1)>\pi(i)<\pi(i+1)$. We say that $\pi$ has $k$ alternating
runs if there are $k-1$ indices $i$ such that $\pi$ changes
direction at these positions. Denote by $R(n, k)$ the number of
permutations in $S_n$ having $k$ alternating runs. Then we have
\begin{eqnarray}\label{relation+run}
R(n, k) = k\,R(n- 1, k)+2R(n-1, k -1) + (n- k)R(n-1, k-2)
\end{eqnarray}  for $n\geq2$ and $k\geq1$ where $R(2,1) = 2$ and $R(n,k) =
0$ unless for $1\leq k\leq n-1$, see Andr\'{e} \cite{An84}. For
$n\geq2$, we have the alternating-runs polynomial $R_n(x) =
\sum_{k=0}^{n}R(n, k)x^k$. Beginning with Andr\'{e} \cite{An84},
there are many interesting results related to alternating runs, see
B\'{o}na and Ehrenborg \cite{BE00}, Canfield and Wilf \cite{CW08},
David and Barton \cite{DB}, Stanley \cite{Sta08}, Zhu \cite{Zhu18E},
Zhu-Yeh-Lu \cite{ZYL}, Zhuang \cite{Zhuang16} for instance. If let
$\widetilde{R}(n,k)=R(n+2,k+1)/2$ for $n,k\geq0$, then the array
$[\widetilde{R}(n,k)]_{n,k}$ satisfies the recurrence relation
\begin{eqnarray}\label{relation+run+shift}
\widetilde{R}(n,k) = (k+1)\widetilde{R}(n- 1, k)+2\widetilde{R}(n-1,
k -1) + (n- k+1)\widetilde{R}(n-1, k-2)
\end{eqnarray}  for $n\geq k\geq0$, where $\widetilde{R}(0,0) = 1$ and $\widetilde{R}(n,k) =
0$ unless for $0\leq k\leq n$.

Let $\{a_1,a_2,c,d\}\subseteq \mathbb{R^{+}}$ and $\{b_2\}\subseteq
\mathbb{R}$. Let a triangle $[T_{n,k}]_{n,k}$ satisfy the next
recurrence relation:
\begin{equation}\label{Recurece+T}
T_{n,k}=d(a_1k+a_2)T_{n-1,k}+b_2T_{n-1,k-1}+c(n-k+1)T_{n-1,k-2},
\end{equation}
where $T_{0,0}=1$ and $T_{n,k}=0$ unless $0\le k\le n$. Denote its
row-generating function by $T_n(x)=\sum_{k=0}^nT_{n,k}x^k$. We
present a generalization of (\ref{relation+run+Eulerian}) as
follows.

\begin{thm}\label{thm+G+runs+Euler}
Let the generalized Eulerian polynomial $\mathcal{E}_{n}(x)$ and the
polynomial $T_{n}(x)$ be defined by (\ref{rec+SEF}) and
(\ref{Recurece+T}), respectively. Let $b_2=d(a_2+\sigma a_1)$. If
$c=da_1$ and $\sigma\in\{0,1,-1\}$, then
\begin{eqnarray*}
T_{n}(x)&=&a^{-\sigma}_2(d+dx)^n\left(\frac{1+w}{2}\right)^{n+\sigma}\mathcal
{E}_{n+\sigma}\left(\frac{1-w}{1+w}\right)
\end{eqnarray*}
for $n\geq1$, where $w=\sqrt{\frac{1-x}{1+x}}$.
\end{thm}

\begin{proof}
For $[\mathcal {E}_{n,k}]_{n,k\geq0}$, by Theorem \ref{thm+SEF},
there exists a nonnegative triangle $[\mathcal {S}_{n,k}]_{n,k}$
satisfying the recurrence relation
\begin{equation}\label{rec+SS}
\mathcal {S}_{n,k}=(a_1k+a_2)\mathcal
{S}_{n-1,k}+da_1(n-2k+1)\mathcal {S}_{n-1,k-1},
\end{equation}
where $\mathcal {S}_{0,0}=1$ and $\mathcal {S}_{n,k}=0$ unless $0\le
k\le (n+1)/2$ such that for $n\geq1$ their row-generating functions
\begin{eqnarray*}
\mathcal {E}_n(x)=(1+x)^n\mathcal {S}_n\left(\frac{\frac{2}{d}
x}{(1+x)^{2}}\right),
\end{eqnarray*}
which implies
\begin{eqnarray}\label{relation+function+S+A}
\left(\frac{1+w}{2}\right)^n\mathcal {E}_n\left(\frac{1-w}{1+w}\right)&=&\left(\frac{1+w}{2}\right)^n\left(\frac{2}{1+w}\right)^n\mathcal {S}_n\left(\frac{1-w^2}{2d}\right)\nonumber\\
&=&\mathcal {S}_n\left(\frac{x}{d(1+x)}\right).
\end{eqnarray}

On the other hand, for the array $[T_{n,k}]_{n,k}$, it follows from
Theorem \ref{thm+m+tran} with $a_0=b_0=b_1=0$ and $\lambda=d$, and
$c=da_1$ that there exists a triangle $[A_{n,k}]_{n,k\geq0}$
satisfying the recurrence relation
 \begin{eqnarray}\label{rec+S}
A_{n,k} =(a_1k+a_2)A_{n-1,k}+da_1(n-2k+1+\sigma)A_{n-1,k-1}
\end{eqnarray}
with $A_{0,0}=1$ and $A_{n,k}=0$ unless $0\le k\le (n+1+\sigma)/2$
such that row-generating functions
\begin{eqnarray}\label{rel+eq}
T_n(x)=(d+dx)^nA_n\left(\frac{x}{d +dx}\right).
\end{eqnarray}

(i) Assume $\sigma=0$. Obviously, $\mathcal {S}_{n,k}=A_{n,k}$ and
$\mathcal {S}_{n}(x)=A_n(x)$ for $n,k\geq0$. Combining identities
(\ref{relation+function+S+A}) and (\ref{rel+eq}), we have for
$n\geq0$ that
\begin{eqnarray*}
T_n(x)&=&(d+dx)^{n}\left(\frac{1+w}{2}\right)^n\mathcal
{E}_n\left(\frac{1-w}{1+w}\right).
\end{eqnarray*}

(ii) Assume $\sigma=-1$. Obviously, $\mathcal
{S}_{n,k}=A_{n+1,k}/a_2$ and $A_{n+1}(x)=a_2\mathcal {S}_{n}(x)$ for
$n,k\geq0$. Thus, combining identities (\ref{relation+function+S+A})
and (\ref{rel+eq}), we have for $n\geq1$ that
\begin{eqnarray*}
T_{n}(x)&=&a_2(d+dx)^{n}\left(\frac{1+w}{2}\right)^{n-1}\mathcal
{E}_{n-1}\left(\frac{1-w}{1+w}\right).
\end{eqnarray*}

(iii) Assume $\sigma=1$. Obviously, $A_{n,k}=\mathcal
{S}_{n+1,k}/a_2$ and $\mathcal {S}_{n+1}(x)=a_2A_{n}(x)$ for
$n,k\geq0$. Thus, combining identities (\ref{relation+function+S+A})
and (\ref{rel+eq}), we have for $n\geq0$ that
\begin{eqnarray*}
T_{n}(x)&=&a^{-1}_2(d+dx)^n\left(\frac{1+w}{2}\right)^{n+1}\mathcal
{E}_{n+1}\left(\frac{1-w}{1+w}\right).
\end{eqnarray*}
This proof is complete.
\end{proof}

\section{Applications}

In this section, we will give some applications of our results.

\subsection{The array from the Lambert function}

Our results can immediately be applied to the array
$[\beta_{n,k}]_{n,k\geq0}$ in Subsection $1.4$.
\begin{prop}
\begin{itemize}
\item [\rm (i)]
Each row $\{\beta_{n,k}\}_{k=0}^n$ is positive;
\item [\rm (ii)]
Each row $\{\beta_{n,k}\}_{k=0}^n$ is log-concave;
\item [\rm (iii)]
Its row generating functions $\beta_{n}(q)$ form a strongly
$q$-log-convex sequence.
\end{itemize}
\end{prop}
\begin{proof}
Let $T_{n,k}=\beta_{n+1,k}$ for $n,k\geq0$. Then by the
recurrence relation (\ref{rec+lab+array}), we have the array
$[T_{n,k}]_{n,k\geq0}$ satisfies the recurrence relation
\begin{eqnarray*}
T_{n,k}=(3n-k-1)T_{n-1,k}+nT_{n-1,k-1}-(k+1)T_{n-1,k+1}
\end{eqnarray*}
for $n,k\geq0$ with $T_{0,0}=1$. For its reciprocal array
$[T^*_{n,k}]_{n,k\geq0}$, it satisfies the recurrence relation
\begin{eqnarray*}
T^*_{n,k}=nT^*_{n-1,k}+(2n+k-1)T^*_{n-1,k-1}-(n-k+1)T^*_{n-1,k-2}
\end{eqnarray*}
for $n,k\geq0$ with $T^*_{0,0}=1$. By Theorem \ref{thm+m+tran} with
$c=-1$ and $d=\lambda$, there exists an array $[A_{n,k}]_{n,k\geq0}$
satisfying the recurrence relation
\begin{eqnarray*}
A_{n,k}=\frac{n}{d}A_{n-1,k}+(n+k-1)A_{n-1,k-1}
\end{eqnarray*}
for $n,k\geq0$ with $A_{0,0}=1$ such that
$$T^*_n(x)=d^n(1+x)^nA_n(\frac{x}{d(1+x)})$$
for $n\geq1$. It is obvious that $\{A_{n,k}\}_{k=0}^n$ is positive
for $d>0$. Thus, $\{T^*_{n,k}\}_{k=0}^n$ is positive
 and log-concave for $0\leq k\leq n$ by Proposition
 \ref{prop+m+tran} (In fact, for $d=1$, $A_{n}(x)$ is exactly the Ramanujan
 polynomial \cite{LZ14}).

 In addition, it follows from \cite{CWY11,Zhu18E} that the row-generating functions $A_n(q)$ form a
 strongly $q$-log-convex sequence for $n\geq0$. Thus $\{A^{*}_n(q)\}_{n\geq0}$ is a strongly $q$-log-convex
 sequence. Noting for the generating function $\beta_{n}(q)$ of the
 array $[\beta_{n,k}]_{n,k\geq0}$ that $$\beta_{n+1}(q)=T_n(q)=A_n^{*}(d(q+1))$$ for $n\geq0$, we immediately get that $\{\beta_n(q)\}_{n\geq0}$ is a strongly
$q$-log-convex sequence.
\end{proof}

\begin{rem}
For a polynomial sequence $\{f_n(q)\}_{n\geq 0}$, it is called {\it
$q$-log-convex} if
$$ f_{n+1}(q)f_{n-1}(q)- f_n(q)^2$$ is a polynomial with nonnegative coefficients for $n\geq 1$, see Liu and Wang \cite{LW07}.
It is called {\it strongly $q$-log-convex} defined by Chen  et al.
\cite{CWY10}, if
$$f_{n+1}(q)f_{m-1}(q)- f_n(q)f_m(q)$$
is a polynomial with nonnegative coefficients for any $n\geq
m\geq1$. Many famous polynomials were proved to be $q$-log-convex or
strongly $q$-log-convex, e.g., the Bell polynomials, the classical
Eulerian polynomials, the Narayana polynomials of type $A$ and $B$,
Jacobi-Stirling polynomials, and so on, see Liu and Wang
\cite{LW07}, Chen et al. \cite{CWY10}, Zhu
\cite{Zhu13,Zhu18E,Zhu18}, Zhu et al. \cite{ZYL}, Zhu and Sun
\cite{ZS15} for instance.
\end{rem}
\subsection{Staircase tableaus}

We will apply our results to the array $\mathscr{T}_{n,k}$ from
staircase tableaus and it generating function $\mathscr{T}_n(q)$ in
subsection $1.3$.

\begin{prop}
Let $\mathscr{T}_{n,k}$ and $\mathscr{T}_n(q)$ be defined in
subsection $1.3$. Then
\begin{itemize}
\item [\rm (i)]
The polynomial $\mathscr{T}_n(q)$ has only real zeros and is
log-concave.
\item [\rm (ii)]
An explicit formula for $\mathscr{T}_{n,m}$ can be written as
\begin{eqnarray*}
\mathscr{T}_{n,m}&=&\sum_{k,i,j\geq0}\binom{n-k}{m-k}\binom{n-i}{k-i}\binom{\frac{1}{2}+i}{i}\binom{i}{j}(-2)^{k}(-1)^{j}(1+j)^n.
\end{eqnarray*}
\item [\rm (iii)]
The exponential generating function of $\mathscr{T}_n(q)$ can be
expressed as
\begin{equation*}
\sum_{n\geqslant
0}\frac{\mathscr{T}_n(q)}{n!}t^n=\left(\frac{(q-1)e^{(q-1)t/3}}{2q-(q+1)
e^{(q-1)t}}\right)^{3/2}.
\end{equation*}
\end{itemize}
\end{prop}

\begin{proof}
(i) By Theorem \ref{thm+m+tran}, there exists an array
$[A_{n,k}]_{n,k}$ satisfying the recurrence relation
$$A_{n,k}=(1+k)A_{n-1,k}+(2n-2k+1)A_{n-1,k-1}$$ with $A_{0,0}=1$ such that
\begin{eqnarray}\label{rel+T+flower}
\mathscr{T}_n(q)=(1+q)^{n}A_n\left(\frac{q}{1+q}\right)
\end{eqnarray}
for $n\geq0$, where $A_n(q)=\sum_{k\geq0} A_{n,k}q^k$. In addition,
$A_n(x)$ satisfies the recurrence relation
\begin{eqnarray*}A_n(q)
&=&[1+(2n-1)q]A_{n-1}(q)+x(1-2x)A'_{n-1}(q)
\end{eqnarray*}
with $A_0(q)=1$. By Lemma \ref{lem-LW}, $A_n(q)$ has only real
zeros. It follows from (\ref{rel+T+flower}) that $T_n(q)$ has only
real zeros. Therefore $\mathscr{T}_n(q)$ is log-concave for
$n\geq0$. (In fact, the triangle $[A_{n,k}]_{n,k\geq0}$ is called
the flower triangle, see \cite[A156920]{Slo}.)

(ii) For array $[A_{n,k}]_{n,k}$, it follows from Theorem
\ref{thm+Stirling+family} (i), we have
\begin{eqnarray*}
A_{n,k}&=&\sum_{i\geq0}\mathcal
{F}_{n,i}\binom{n-i}{k-i}(-2)^{k-i},\\
 \mathcal
{F}_{n,i}&=&\sum_{j\geq0}\binom{\frac{1}{2}+i}{i}\binom{i}{j}2^{i}(-1)^{i-j}(1+j)^n.
\end{eqnarray*}

Thus by Theorem \ref{thm+m+tran}, we get for $n,m\geq0$ that
\begin{eqnarray*}
\mathscr{T}_{n,m}=\sum_{k\geq0}A_{n,k}\binom{n-k}{m-k}=\sum_{k,i,j\geq0}\binom{n-k}{m-k}\binom{n-i}{k-i}\binom{\frac{1}{2}+i}{i}\binom{i}{j}(-2)^{k}(-1)^{j}(1+j)^n.
\end{eqnarray*}

(iii) Moreover, by (\ref{EGF+F}) and Theorem \ref{thm+EF}, we deduce
the expression
\begin{equation*}
\sum_{n\geqslant
0}\frac{A_n(q)}{n!}t^n=\left(\frac{(2q-1)e^{(2q-1)t/3}}{2q-
e^{(2q-1)t}}\right)^{3/2},
\end{equation*}
which by (\ref{rel+T+flower}) implies the desired result in (iii).
\end{proof}

\subsection{Alternating runs of type $B$}

A {\it run} of a signed permutation $\pi$ is defined as a maximal
interval of consecutive elements on which the elements of $\pi$ are
monotonic in the order
$\cdots<\overline{2}<\overline{1}<0<1<2<\cdots$. The {\it up signed
permutations} are signed permutations with $\pi(1)> 0$. Let $Z(n,k)$
denote the number of up signed permutations in $B_n$ with $k$
alternating runs. In \cite[Theorem 4.2.1]{Zhao11}, Zhao demonstrated
that the numbers $Z(n,k)$ satisfy the recurrence relation
\begin{equation}\label{rec+Run+B}
Z(n,k)=(2k-1)Z(n-1,k)+3Z(n-1,k-1)+(2n-2k+2)Z(n-1,k-2)
\end{equation}
for $n\geqslant 2$ and $1\leqslant k\leqslant n$, where $Z(1,1)=1$
and $Z(1,k)=0$ for $k>1$. The alternating runs polynomials of type
$B$ are defined by $Z_n(x)=\sum_{k=1}^nZ(n,k)x^k$. The first few of
the polynomials $Z_n(x)$ are given as follows:
\begin{align*}
Z_1(x)=x,\,
 Z_2(x)=x+3x^2,\,
  Z_3(x)=x+12x^2+11x^3.
\end{align*}
Zhao proved that $Z_n(x)$ has only real zeros and is log-concave.
Let $\widetilde{Z}(n,k)=Z(n+1,k+1)$ for $n,k\geq0$. Then it follows
from (\ref{rec+Run+B}) that
\begin{equation}\label{tnk-recurrence04}
\widetilde{Z}(n,k)=(2k+1)\widetilde{Z}(n-1,k)+3\widetilde{Z}(n-1,k-1)+2(n-k+1)\widetilde{Z}(n-1,k-2)
\end{equation}
for $n\geq1$ with $\widehat{Z}(0,0)=1$. Using our results, we
immediately get the following properties.

\begin{prop}
Let $\widetilde{Z}_n(x)$ be defined above. We have the following
results.
\begin{itemize}
\item [\rm (i)]
The polynomial $\widetilde{Z}_n(x)$ has only real zeros and is
log-concave.
\item [\rm (ii)] For $n\geq0$, we have
\begin{eqnarray*}
\widetilde{Z}_n(x)=(1+x)^{n}W^l_{n+1}\left(\frac{2x}{1+x}\right),
\end{eqnarray*}where $W^l_{n}(x)=\sum_{k\geq0}W^l_{n,k}x^{k}$ and $W^l_{n,k}$
satisfies the recurrence relation
\begin{eqnarray}\label{re+W^l}
W^l_{n,k}=(2k+1)W^l_{n-1,k}+(n-2k+1)W^l_{n-1,k-1}
\end{eqnarray}
with initial condition $W^l_{0,0}=1$, see Petersen \cite{Pe07}.
\item [\rm (iii)]
Let $ \mathfrak{B}_n(x)$ be the Eulerian polynomials of type $B$ in
subsection $1.2$. Then we have a David-Barton type formula
\begin{eqnarray*}
\widetilde{Z}_{n}(x)=(1+x)^n\left(\frac{1+w}{2}\right)^{n+1}\mathfrak{B}_{n+1}\left(\frac{1-w}{1+w}\right).
\end{eqnarray*}
\item [\rm (iv)]
Let $Q_n(\delta,t)$ be the general derivative polynomial. Then for
$n\geq0$ we have
\begin{eqnarray*}
\widetilde{Z}_n(-x)&=&2^n(1-x)^n\left(\frac{1}{\sqrt{x-1}}\right)^{n+1}Q_{n+1}\left(1,\sqrt{x-1}\right).
\end{eqnarray*}
\end{itemize}
\end{prop}

\begin{proof}
(i) and (ii)  Applying Theorem \ref{thm+m+tran} with $c=2$ and
$\lambda=d=1$ to $\widetilde{Z}(n,k)$, there exists an array
$[A_{n,k}]_{n,k\geq0}$ satisfying the recurrence relation
$$A_{n,k}=(2k+1)A_{n-1,k}+2(n-2k+2)A_{n-1,k-1}$$
for $n\geq1$ with $A_{0,0}=1$ such that
\begin{eqnarray*}
\widetilde{Z}_n(x)=(1+x)^{n}A_n\left(\frac{x}{(1+x)}\right)
\end{eqnarray*}
for $n\geq0$. Note that
$A_n(x)=(1+2x)A_{n-1}(x)+2x(1-2x)A'_{n-1}(x)$ for $n\geq1$. By Lemma
\ref{lem-LW}, $A_n(x)$ has only real zeros. Thus
$\widetilde{Z}_n(x)$ has only real zeros and is log-concave.

Then by the recurrence relation (\ref{re+W^l}) we have
$A_{n,k}=W^l_{n+1,k}2^k$ for $n,k\geq0$. Thus we get
\begin{eqnarray*}
\widetilde{Z}_n(x)=(1+x)^{n}W^l_{n+1}\left(\frac{2x}{1+x}\right)
\end{eqnarray*}
for $n\geq0$ and $d=2$.

For (iii), it follows from Theorem \ref{thm+G+runs+Euler} with
$\sigma=1$.

For (iv), it follows from (i) and Theorem \ref{thm+Andr}.
\end{proof}




\end{document}